\documentclass{amsart}
\usepackage[unicode=true,bookmarks=true,bookmarksnumbered=false,bookmarksopen=false, breaklinks=false,pdfborder={0 0 1},backref=false,colorlinks=false] {hyperref}
\hypersetup{pdfauthor={Giosu{\`e} Muratore}}
\usepackage{comment}
\usepackage{enumitem}

\numberwithin{equation}{section}
\numberwithin{figure}{section}
\numberwithin{table}{section}
\theoremstyle{plain} 
\newtheorem{thm}{\protect\theoremname}[section]
\newtheorem{prop}[thm]{\protect\propositionname}
\newtheorem{cor}[thm]{\protect\corollaryname}
\newtheorem{lem}[thm]{\protect\lemmaname}
\newtheorem{conj}[thm]{\protect\conjecturename}

\theoremstyle{definition} 
\newtheorem{defn}[thm]{\protect\definitionname}
\newtheorem{example}[thm]{\protect\examplename}

\theoremstyle{remark} 
\newtheorem{rem}[thm]{\protect\remarkname}
\newtheorem*{ack}{\protect\acknowledgmentkname}

\usepackage{tikz}
\usetikzlibrary{positioning}
\tikzset{every picture/.style={line width=0.75pt}} 

\makeatother

\providecommand{\corollaryname}{Corollary}
\providecommand{\definitionname}{Definition}
\providecommand{\examplename}{Example}
\providecommand{\propositionname}{Proposition}
\providecommand{\remarkname}{Remark}
\providecommand{\theoremname}{Theorem}
\providecommand{\lemmaname}{Lemma}
\providecommand{\problemname}{Problem}
\providecommand{\acknowledgmentkname}{Acknowledgements}
\providecommand{\claimname}{Claim}
\providecommand{\conjecturename}{Conjecture}

\numberwithin{equation}{section}

\newcommand{\abs}[1]{\lvert#1\rvert}
\newcommand{\ev}{\mathrm{ev}}
\newcommand{\pol}{T}

\newcommand{\blankbox}[2]{%
  \parbox{\columnwidth}{\centering
    \setlength{\fboxsep}{0pt}%
    \fbox{\raisebox{0pt}[#2]{\hspace{#1}}}%
  }%
}

\begin{document}

\title{Relative Gromov--Witten and maximal contact conics}


\author[Muratore]{Giosu{\`e} Muratore}
\address{CEMS.UL (University of Lisbon), and  
COPELABS/DEISI (Lus\'ofona University)}
\email{muratore.g.e@gmail.com}
\urladdr{sites.google.com/view/giosue-muratore}





\begin{abstract}
We discuss some properties of the relative Gromov--Witten invariants counting rational curves with maximal contact order at one point. We compute the number of Cayley's sextactic conics to any smooth plane curve. In particular, we compute the contribution, from double covers of inflectional lines, to a certain degree two relative Gromov--Witten invariant relative to the curve.
\end{abstract}

\maketitle

    
\section{Introduction}
Let $Y$ be a complex general smooth plane curve of degree $d\ge3$. By the classical B{\'e}zout theorem, a rational curve of degree $n$ meets $Y$ at $dn$ points. If we require that the contact order at some point is $3n$ (the maximal possible), we expect that the number of such rational curves is finite. For example, if $n=1$ these curves are the inflectional lines.

When $n=2$, they are called sextactic conics, and the relevant intersection point is called sextactic point. The main result of this paper is the following. 
\begin{thm}\label{thm:main}
    Let $Y$ be a general smooth plane curve of degree $d\ge 3$. The number of conics of contact order $6$ with $Y$ is
    $$n_d=3d(4d-9).$$
\end{thm}
This result was first found by \cite{cayley1863sextactic}, see also \cite{MR1924719} for another proof and historical background.

Our proof uses relative Gromov--Witten invariants, and closely follows the approach of \cite{MR2134277} (where the number of conics five-fold tangent to $Y$ is computed). We proceed as follows: In Section \ref{section:explicit}, we compute explicitly some relative Gromov--Witten invariants relative to a smooth hypersurface of $\mathbb P^s$. In particular, in Equation~\eqref{eq:max_lines}, we give a closed formula for the number of lines with maximal contact with a degree $d$ hypersurface\footnote{A different version of this formula appears in \cite{bertone}.}. Finally, in Section \ref{section:sextactic}, we prove the theorem as a difference between virtual numbers of stable maps tangent to the curve $Y$. The central point of the proof is the computation of the contribution from double covers of lines to a relative Gromov--Witten invariant. In the last subsection we conjecture an analogous polynomial counting curves of degree $n=3$, using the multiple cover contribution formula of Gross, Pandharipande, and Siebert.

When $Y$ is cubic, rational curves with high contact are well understood \cite{MR2435425}. 
Moreover, $Y$ meets a maximal contact curve at just one point. This point corresponds to $3n$-torsion points of the elliptic curve $(Y,p)$, where $p$ is inflectional \cite[Lemma~1.2]{taka}. Thus, sextactic points correspond to $6$-torsion points that are not $3$-torsion. Moreover, the fact that the pair $(\mathbb P^2, Y)$ is log Calabi--Yau is used in \cite{MR1844627,MR4513164}. 
See \cite{MR4371043,MR4675044,MR4586769} for related recent results. One of the purposes of this project is to have a better comprehension of curves of maximal contact with $Y$ when $d>3$.

In the future, we plan to further expand our result either by increasing the value of $n$ (as in Conjecture~\ref{conj:n3}), or by considering curves of maximal contact to a surface $Y\subset\mathbb P^3$. 
\begin{ack}
    I thank Ethan Cotterill for introducing me to the problem. I also thank Miguel Moreira, Georg Oberdieck and Israel Vainsencher for useful discussions. The author is a member of GNSAGA (INdAM). The author is supported by FCT - Funda\c{c}\~{a}o para a Ci\^{e}ncia e a Tecnologia, under the project: UID/04561/2025.
\end{ack}
\section{Notation and background}
Let $\overline M_{0,k}(\mathbb P^s,n)$ be the moduli space of degree $n$ rational stable maps with $k$ marked point, see \cite[0.4]{FP}. We denote by $\ev_j\colon\overline M_{0,k}(\mathbb P^s,n)\rightarrow \mathbb P^s$ the $j^\mathrm{th}$ evaluation map, or simply by $\ev$ the unique evaluation map when $k=1$. Let $Y\subset\mathbb P^s$ be a smooth hypersurface of degree $d$, and $H\subset\mathbb P^s$ be a hyperplane. We denote by $\overline M_{(m)}^Y(\mathbb P^s,n)$ the moduli space of $1$-pointed relative
stable maps of degree $n$ to $\mathbb P^s$ relative to $Y$ with multiplicity $m$ at the marked point, as in Definition~1.1 and 1.18 of \cite{MR1944571}. This space can be thought as the compactification of the moduli space of maps whose contact order at the marked point is at least equal to $m$. It has expected dimension
$$
\dim(\overline M_{0,1}(\mathbb P^s,n))-m=s+n(s+1)-2-m.
$$
The virtual fundamental class, denoted by $[\overline M_{(m)}^Y(\mathbb P^s,n)]$, satisfies the following recursive relation for every $m\ge0$,
\begin{equation}\label{eq:fundamental}
        (\ev^*(Y)+m\psi)\cdot [\overline M^Y_{(m)}(\mathbb P^s, n)]=[\overline M^Y_{(m+1)}(\mathbb P^s, n)]+[D_{(m)}(\mathbb P^s, n)],
\end{equation}
where $D_{(m)}(\mathbb P^s, n)$ is a space that parameterizes reducible curves. 
We refer to \cite[Theorem 0.1]{MR1962055} for all details of Equation~\eqref{eq:fundamental}. 
When $n=1$ and $Y$ has no rational curves, there is no excess locus $D_{(m)}(\mathbb P^s,1)$. But when $n>1$, the situation is more complicated.
\begin{example}\label{example:my_example}
    Let $s=n=2$. If $\deg(Y)\ge3$, then $Y$ has no rational curves. It follows that
    $$(\ev^*(Y)+5\psi)\cdot [\overline M^Y_{(5)}(\mathbb P^2, 2)]=[\overline M^Y_{(6)}(\mathbb P^2, 2)]+[D_{(5)}(\mathbb P^2, 2)],$$
    where $[D_{(5)}(\mathbb P^2, 2)]=3[D(\mathbb P^2, (1,1),(2,3))]$ and
    $$
    D(\mathbb P^2, (1,1),(2,3))=\overline{M}_{0,3}(Y,0)\times_Y \overline{M}^Y_{(2)}(\mathbb P^2, 1)\times_Y \overline{M}^Y_{(3)}(\mathbb P^2, 1).
    $$
\end{example}
Now, we define the virtual number of maximal contact curves. In the next section we compute such a degree in some particular cases. 
\begin{defn}\label{defn:virt_num}
    Let $s,n$ be positive integers. We denote by $\pol_{s,n}\colon\mathbb{N}\rightarrow \mathbb{Q}$ the virtual number of maximal contact curves of degree $n$, that is the function
    \begin{equation}\label{eq:Tsn}
        d\mapsto \deg([\overline M_{(s-2+n(s+1))}^Y(\mathbb{P}^s,n)]),
    \end{equation}
    where $Y$ denotes a general hypersurface of degree $d$.
\end{defn}
\section{Explicit computation of some relative GW invariants}
\label{section:explicit}
\subsection{The case of lines}
For completeness, in this section we deduce a closed formula for the number of lines of maximal contact with a hypersurface $Y$.
We adopt the following convention of the binomial coefficients:
$$
\binom a i=0 \,\,\mathrm{if}\,\, a<i.
$$
Moreover, $\psi$ will always be the $\psi$-class with respect to the first mark.
\begin{lem}
Let $a,i$ be two non-negative integers. We have the following,
\begin{align}\label{first}
\int_{\overline{M}_{0,2}(\mathbb P^s,1)} \psi^{a}\ev_{1}^*(H)^{s-a+i}\ev_{2}^*(H)^{s-i}&=
(-1)^{a-i}\binom a i ,\\
\label{second}\int_{\overline{M}_{0,1}(\mathbb P^s,1)} \psi^{a}\ev_{}^*(H)^{2s-1-a} &=
(-1)^{a-s-1}\binom {a+1} s.
\end{align}
\end{lem}
\begin{proof}
    We prove the first equation by induction on the sum $a+i=t$. If $t=0$, then the only case we need to prove is $a=i=0$. That is, we need to prove that
    $$
    \int_{\overline{M}_{0,2}(\mathbb P^s,1)} \psi^{0}\ev_{1}^*(H)^{s}\ev_{2}^*(H)^{s}=1.
    $$
    But that is of course true since it is the number of lines passing through two general points \cite[Lemma 14]{FP}. Let us suppose that the equation holds true whenever $a+i\leq t-1$. Let $a,i$ be such that $a+i=t$. Consider the invariant
    \begin{equation}\label{invariant}
        \int_{\overline{M}_{0,3}(\mathbb P^s,1)} \psi^{a}\ev_{1}^*(H)^{s-a+i}\ev_{2}^*(H)^{s-i}\ev_{3}^*(H).
    \end{equation}
    Using the divisor equation \cite[1.2.III]{MR1685628}, we see that
    \begin{align*}
    \int_{\overline{M}_{0,3}(\mathbb P^s,1)} &\psi^{a}\ev_{1}^*(H)^{s-a+i}\ev_{2}^*(H)^{s-i}\ev_{3}^*(H)\\
          =  \int_{\overline{M}_{0,2}(\mathbb P^s,1)} &\psi^{a}\ev_{1}^*(H)^{s-a+i}\ev_{2}^*(H)^{s-i}+  \psi^{a-1}\ev_{1}^*(H)^{s-(a-1)+i}\ev_{2}^*(H)^{s-i}\\
          =  \int_{\overline{M}_{0,2}(\mathbb P^s,1)} &\psi^{a}\ev_{1}^*(H)^{s-a+i}\ev_{2}^*(H)^{s-i}+(-1)^{a-1-i}\binom{a-1}i.
    \end{align*}
    Note that we used the inductive hypothesis. On the other hand, using the topological recursion relation \cite[Equation~(6)]{MR1685628}, we see that \eqref{invariant} is equal to
    \begin{equation*}
        \int_{\overline{M}_{0,2}(\mathbb P^s,1)} \psi^{a-1}\ev_{1}^*(H)^{s-a+i}\ev_{2}^*(H)^{s-(i-1)} =  (-1)^{a-i}\binom{a-1}{i-1}.
    \end{equation*}
    We deduce immediately
    \begin{align*}
        \int_{\overline{M}_{0,2}(\mathbb P^s,1)} \psi^{a}\ev_{1}^*(H)^{s-a+i}\ev_{2}^*(H)^{s-i} &= (-1)^{a-i}\binom{a-1}i +(-1)^{a-i}\binom{a-1}{i-1}\\
        &= (-1)^{a-i}\binom{a}{i}.
    \end{align*}
    So, \eqref{first} is proved. In order to prove \eqref{second}, it is enough to take $i=s$ in \eqref{first} and apply the string equation \cite[1.2.I]{MR1685628}.
\end{proof}
Let $\mathcal{S}(n,k)$ denote the signed Stirling numbers of the first kind. More precisely, $\mathcal{S}(n,k)$ are the integers such that
$$x(x+1)(x+2)\cdots(x+n-1)=\sum_{k=0}^n(-1)^{n-k}\mathcal{S}(n,k)x^k.$$
Using Equations~\eqref{eq:fundamental} and \eqref{second}, we have
\begin{align*}
    \deg([\overline{M}^Y_{(2s-1)}(\mathbb P^s, 1)]) &= \sum_{k=0}^{2s-1} (-1)^{2s-1-k}\mathcal{S}(2s-1,k)\int_{\overline{M}_{0,1}(\mathbb P^s,1)} \psi^{2s-1-k}\ev_{}^*(H)^{k} d^k\\
    &= \sum_{k=1}^{s} (-1)^{k+1}\mathcal{S}(2s-1,k)\int_{\overline{M}_{0,1}(\mathbb P^s,1)} \psi^{2s-1-k}\ev_{}^*(H)^{k} d^k\\
    &= \sum_{k=1}^{s}(-1)^{k+1+2s-1-k-s-1}\mathcal{S}(2s-1,k)\binom{2s-k}{s}d^k.
\end{align*}
Finally,
\begin{equation}\label{eq:max_lines}
    \deg([\overline{M}^Y_{(2s-1)}(\mathbb P^s, 1)]) = (-1)^{s+1}\sum_{k=1}^{s}  \mathcal{S}(2s-1,k)\binom{2s-k}{s}d^k.
\end{equation}

\begin{example}\label{example:Infl}
    Take $s=2$ and $s=3$. We easily get
    \begin{align*}
        \mathcal{S}(3,1) &=2, & \mathcal{S}(3,2)&=-3, & & \\
        \mathcal{S}(5,1) &=24, & \mathcal{S}(5,2)&=-50, & \mathcal{S}(5,3)&=35 .
    \end{align*}
    So we have
    \begin{align*}
        \deg([\overline{M}^Y_{(3)}(\mathbb P^2, 1)]) &= 3d(d-2),\\
        \deg([\overline{M}^Y_{(5)}(\mathbb P^3, 1)]) &= 35d^3-200d^2+240d.
    \end{align*}
    Both cases are classically known, see \cite[Exercise~IV.2.3(e)]{MR0463157} and \cite[11.1.3]{MR3617981}. 
\end{example}

\subsection{The case of any degree}
When the degree of the curve is greater than $1$, a closed formula like Equation~\eqref{eq:max_lines} is not known. But the following proposition is very useful, because it allows to completely determine relative invariants after knowing a finite number of them. We denote by $p_t(d)\in\mathbb Q[d]$ a polynomial of degree at most $t$.
\begin{prop}\label{prop:polynomial}
    Let $Y\subset \mathbb P^s$ be a smooth hypersurface of degree $d$, which does not contain rational curves. For any $m\ge0$ and $n\ge1$,
    \begin{equation}\label{eq:P}
        \int_{\overline{M}_{(m+1)}^Y(\mathbb{P}^s,n)} \ev^*(Y)^k\psi^j=
        \begin{cases}
            0, & \text{if } j\neq s+n(s+1)-m-k-3,\\
            0, & \text{if } k\ge s,\\
            d^{k+1}p_{s-1-k}(d), & \text{if } 0\le k\le s-1.
        \end{cases}
    \end{equation}
    In particular, the integral in the equation is either zero or a polynomial in $d$ of degree at most $s$ where $0$ is a root of multiplicity $k+1$.
\end{prop}
\begin{proof}
    We prove the proposition by induction on $m$. By Equation~\eqref{eq:fundamental},
    $$
    \int_{\overline{M}_{(1)}^Y(\mathbb{P}^s,n)} \ev^*(Y)^k\psi^j =
        d^{k+1}\int_{\overline{M}_{0,1}(\mathbb P^s,n)} \ev^*(H)^{k+1}\psi^j.
    $$
    Thus Equation~\eqref{eq:P} holds, independently of $n$, for $m=0$. Suppose that $m>0$, hence
    \begin{multline*}
        \int_{\overline{M}_{(m+1)}^Y(\mathbb{P}^s,n)} \ev^*(Y)^k\psi^j =
        \int_{\overline{M}_{(m)}^Y(\mathbb{P}^s,n)} \ev^*(Y)^{k+1}\psi^j\\
        +m\int_{\overline{M}_{(m)}^Y(\mathbb{P}^s,n)} \ev^*(Y)^k\psi^{j+1}-\int_{D_{(m)}(\mathbb{P}^s,n)}\ev^*(Y)^k\psi^{j}.
    \end{multline*}
    By induction, we need to compute only the integral on the right. By the explicit formula in \cite[Remark~1.4]{MR1962055}, using that $Y$ has no rational curves, it is
    \begin{equation}\label{eq:int_D}
        \sum_{r\ge2}\int_{\overline{M}_{0,1+r}}\psi^j \sum_{\gamma_1,\ldots,\gamma_r}\left( (Y_{|Y}^k\cup \gamma_1\cup\cdots\cup\gamma_r)\frac{1}{r!}\prod_{i=1}^r m_{i}
        \int_{\overline{M}_{(m_{i})}^Y(\mathbb{P}^s,n_i)} \ev^*(\gamma_i^\vee)\right),
    \end{equation}
    where $m_1+\ldots+m_r=m$, $n_1+\ldots+n_r=n$, the $\gamma_i$ run over a basis of the part of $H^*(Y)\otimes\mathbb Q$ induced by $\mathbb P^s$, and $\gamma_i^\vee$ is the dual of $\gamma_i$ immersed in $H^*(\mathbb P^s)\otimes\mathbb Q$.

    In order to compute \eqref{eq:int_D}, we may suppose that each $\gamma_i$ has codimension $t_i$, that is $\gamma_i=H_{|Y}^{t_i}$. Thus the cap product is
    $$
    Y_{|Y}^k\cup \gamma_1\cup\cdots\cup\gamma_r = 
    d^kH_{|Y}^k\cup \gamma_1\cup\cdots\cup\gamma_r = d^{k+1}
    $$
    if and only if $\sum_{i=1}^r t_i=s-1-k$, otherwise it is zero. Since $m_i<m$ for any $i$, by induction we have
    \begin{align*}
        \prod_{i=1}^r m_{i}
        \int_{\overline{M}_{(m_{i})}^Y(\mathbb{P}^s,n_i)} \ev^*(\gamma_i^\vee) &= \prod_{i=1}^r m_{i}
        \int_{\overline{M}_{(m_{i})}^Y(\mathbb{P}^s,n_i)} \ev^*(\frac{1}{d}H^{s-t_i-1}) \\
        &= \prod_{i=1}^r \frac{m_{i}}{d^{s-t_i}}
        \int_{\overline{M}_{(m_{i})}^Y(\mathbb{P}^s,n_i)} \ev^*(Y^{s-t_i-1}) \\
        &= \prod_{i=1}^r \frac{m_{i}}{d^{s-t_i}}
        d^{s-t_i}p_{t_i}(d) \\
        &= p_{s-1-k}(d).
    \end{align*}
    Thus \eqref{eq:int_D} is either zero or $d^{k+1}p_{s-1-k}(d)$, as expected.
\end{proof}
We may apply Equation~\eqref{eq:P} to Definition~\ref{defn:virt_num}.
\begin{prop}
    There exists a polynomial $p(d)$ of degree at most $s-1$ such that, for every $d\ge 2s-1$,  $\pol_{s,n}(d)=dp(d)$.
\end{prop}
\begin{proof}
    If $d\ge 2s-1$, a general $Y$ is smooth and does not contain rational curves by \cite{MR0875091}. The result follows by Proposition~\ref{prop:polynomial}. 
\end{proof}
Thus, in order to completely determine the polynomial $\pol_{s,n}$, it is enough to compute it for $s$ values of $d$ greater or equal than $2s-1$. Using the packages \cite{MR4383164,MR4739375}, we computed the following values by hand:
\begin{align*}
\pol_{2,2}(3)&=\frac{135}{4}, & \pol_{2,2}(4) &=102,\\
\pol_{2,3}(3)&=244, & \pol_{2,3}(4) &=\frac{2384}{3}.
\end{align*}
Using polynomial interpolation, these values imply that for $d\ge 3$,
\begin{align}
    \pol_{2,2}(d) &= \frac{57}{4}d^2-\frac{63}{2}d, \label{eq:T22} \\
    \pol_{2,3}(d) &= \frac{352}{3}d^2-\frac{812}{3}d.\label{eq:T23}
\end{align}
We checked that \eqref{eq:T22} and \eqref{eq:T23} match \eqref{eq:Tsn} for many values of $d$, using \cite{GROWI}. Using again interpolation, we see that
$$T_{3,2}(d)=\frac{20331}{4}d^3 - 30294 d^2 + 39852 d.$$
In the introduction, we said that we plan to compute the number of maximal contact conics to any surface $Y\subset \mathbb P^3$ of degree $d$. This problem is equivalent to removing from $T_{3,2}(d)$ the contribution from non enumerative stable maps (e.g., double covers).
\begin{rem}
    In Equation~\eqref{eq:P}, setting $m=n(s+1)-2$, $k=s-1$, and $j=0$, the polynomial $p_0(d)$ is constant, thus 
    $$\int_{\overline{M}_{(m+1)}^Y(\mathbb{P}^s,n)} \ev^*(Y)^{s-1}=d^sc$$
    for some $c$. This constant is an integer, see \cite[Proposition~4.1]{MR4256011} for its geometrical meaning. 
\end{rem}
\section{Sextactic conics}
\label{section:sextactic}
In this section we focus on the case of smooth conics. Our goal is to compute the contribution to $T_{2,2}$ from double covers. We denote by $Y$ a general smooth plane curve of degree $d$.
\begin{lem}\label{lem:Gat21}
    Let $m$ be a non negative integer such that $m\le 6$, and let $(C,f,x)$ be an irreducible stable map that is not a double cover of a line. Then the moduli space $\overline M_{(m)}^Y(\mathbb P ^2,2)$ is smooth of dimension $6-m$ at $(C,f,x)$ (which is the expected dimension).
\end{lem}
\begin{proof}
    It follows from \cite[Lemma~2.1]{MR2134277}.
\end{proof}
An immediate consequence of this lemma is the following,
\begin{cor}\label{cor:unique_irr_comp}
    Every irreducible stable map in $\overline M_{(m)}^Y(\mathbb P ^2,2)$ whose image in $\mathbb P^2$ is a smooth
conic lies in a unique irreducible component of $\overline M_{(m)}^Y(\mathbb P ^2,2)$ of the expected dimension. The virtual
fundamental class of this component is equal to the usual one.
Moreover, the number of smooth plane conics of maximal contact with $Y$ is finite. 
\end{cor}
\begin{proof}
    The fact that the stable map $(C,f,x)$ lies in a unique irreducible component of the expected dimension follows directly from the lemma. The virtual fundamental class is equal to the usual one because it is smooth of the right dimension (see \cite[Proposition~5.6]{MR1437495}).

    Finally, the smooth plane conics of maximal contact with $Y$ are supported at the isolated points of $\overline M_{(6)}^Y(\mathbb P ^2,2)$. Hence there are only a finite number of them.
\end{proof}

\begin{prop}\label{prop:M5}
    The moduli space $\overline M^Y_{({5})}(\mathbb P ^2,2)$ is a connected space with the following irreducible components.
    \begin{enumerate}[label=(\Alph*)]
        \item An irreducible component parameterizing smooth conics, denoted by $M^*$.
        \item One smooth rational curve for each line $l$ that is inflectional to $Y$ at $p$. This space parameterizes $1$-marked double covers of $l$ such that the mark is mapped to $p$.
    \end{enumerate}
\end{prop}
\begin{proof}
     If the image of the stable map is a smooth conic, point $(A)$ follows from Corollary~\ref{cor:unique_irr_comp}.
    If the image is instead the union of two distinct lines, then the conditions of \cite[Remark~1.5]{MR2134277} cannot be satisfied. 

    Now, suppose that the stable map is a double cover of a line. If a ramification point is mapped to a point of $Y$, it doubles the contact order of the line with $Y$ at that point. It is clear that the only possibility for getting a contact strictly greater than $4$ is when the ramification point is mapped to an inflectional point. The choice of the other branch point gives a unidimensional family.

    Note that, for each point of $Y$, there is only one conic with contact order $5$. Because if we had two, any conic in the pencil generated by them would have contact order $5$. This is impossible. More formally, we may say that the map of curves $\ev\colon M^*\rightarrow Y$ has degree $1$ by \cite[Proposition~4.1]{MR4256011}.

    Since $Y$ is smooth, we have $M^*\cong Y$. In particular $M^*$ is connected.

    Finally, $\overline M^Y_{({5})}(\mathbb P ^2,2)$ is connected because at any inflectional point $p\in Y$, the only conic with contact order $5$ at $p$ is the double inflectional line. Thus $M^*$ meets any component in $(B)$.
\end{proof}
The components described in point $(B)$ are particularly interesting. We denote them by $M_l$ as in the following.
\begin{defn}
    Let $l$ be a line inflectional to $Y$ at $p$. We denote by $M_l$ the closure of the space of stable maps $(\mathbb P^1,f,x)$ such that $f$ is a double cover of $l$ ramified at $x$ and $f(x)=p$.
\end{defn}
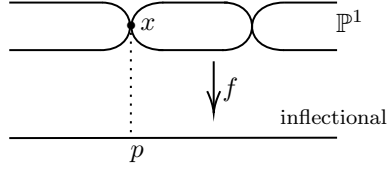
\begin{figure}[t]
    \centering
    \begin{tikzpicture}[x=0.75pt,y=0.75pt,yscale=-1,xscale=1]

\draw    (0.67,82.5) -- (166,82.3) ;
\draw  [dash pattern={on 0.84pt off 2.51pt}]  (62,29.3) -- (62,82.3) ;
\draw    (103.78,43.75) -- (103.78,67.75) ;
\draw [shift={(103.78,69.75)}, rotate = 270] [color={rgb, 255:red, 0; green, 0; blue, 0 }  ][line width=0.75]    (10.93,-3.29) .. controls (6.95,-1.4) and (3.31,-0.3) .. (0,0) .. controls (3.31,0.3) and (6.95,1.4) .. (10.93,3.29)   ;
\draw    (108.98,14.1) .. controls (127.98,14.27) and (127.98,38.52) .. (107.98,38.1) ;
\draw    (47.45,14.1) .. controls (66.45,14.27) and (66.45,38.52) .. (46.45,38.1) ;
\draw    (78.4,14.1) .. controls (57.06,14.21) and (56.75,38.21) .. (78.4,38.1) ;
\draw    (139.25,14.1) .. controls (117.91,14.21) and (117.6,38.21) .. (139.25,38.1) ;
\draw    (0.67,14.1) -- (47.67,14.1) ;
\draw    (0.67,38.1) -- (46.67,38.1) ;
\draw    (75.88,38.1) -- (108.83,38.1) ;
\draw    (77.88,14.1) -- (110.83,14.1) ;
\draw    (138.88,38.1) -- (166,38.1) ;
\draw    (139.25,14.1) -- (166,14.1) ;

\draw (164,16.5) node [anchor=north west][inner sep=0.75pt]    {$\mathbb{P}^{1}$};
\draw (135.8,66.3) node [anchor=north west][inner sep=0.75pt]  [font=\footnotesize] [align=left] {inflectional};
\draw (106.78,50.31) node [anchor=north west][inner sep=0.75pt]    {$f$};
\draw (65.36,20.5) node [anchor=north west][inner sep=0.75pt]    {$x$};
\draw (58.0,22.00) node [anchor=north west][inner sep=0.75pt]    {{${\scriptstyle \bullet}$}};
\draw (60.36,85) node [anchor=north west][inner sep=0.75pt] {$p$};   

\end{tikzpicture}

    \caption{The space $M_l$.}
    \label{fig:Ml}
\end{figure}
Figure~\ref{fig:Ml} gives us an idea of a curve in $M_l$. Before we move on, we need to understand better the virtual fundamental class of $\overline M^Y_{({5})}(\mathbb P ^2,2)$. Note that there exists a section $\sigma$ of a line bundle on $M^*$, whose vanishing describes the condition of increasing by $1$ the contact order at the marked point, see \cite[Remark~1.12]{MR2134277}.
\begin{prop}\label{prop:M5restr_and_section}
    We have the following:
    \begin{enumerate}
        \item The virtual fundamental class of $\overline M^Y_{({5})}(\mathbb P ^2,2)$ on $M_l$ is the usual one.
        \item Let $\mathcal{C}$ be the intersection between $M^*$ and $M_l$. The section $\sigma$ on $M^*$ vanishes at $\mathcal{C}$ with multiplicity $1$.
    \end{enumerate}
\end{prop}
\begin{proof}
    This proof involves some local computations. Let $[z_0:z_1:z_2]$ be homogeneous coordinates for $\mathbb P^2$. We may assume, using an automorphism of $\mathbb P^1$, that a general point of $M_l$ is a double cover ramified at the points $[1:0]$ and $[0:1]$. Moreover, we fix a line $l:z_2=0$ that is inflectional to $Y$ at $[1:0:0]$. Let us consider coordinates of $\overline M_{0,0}(\mathbb P^2,2)$ centered at a double cover of $l$:
    $$
    [s:t]\mapsto [s^2+\epsilon_1 t^2:
    t^2+\epsilon_2s^2:
    \epsilon_3s^2+\epsilon_4st+\epsilon_5t^2].
    $$
    In order to get the space of all double covers that map $[1:0]$ to $[1:0:0]$, it is enough to impose
    $$\epsilon_2=\epsilon_3=\epsilon_4=\epsilon_5=0.$$
    So, we get the one dimensional space of double covers $f\colon\mathbb P^1\rightarrow l$ whose image of a ramification point is fixed. Locally around $[1:0]$, the maps can be represented as
    $$
    f(t)=\frac{t^2}{1+\epsilon_1t^2}.
    $$
    Let us add a marked point whose coordinate is $t=\epsilon$. The tangency condition at the marked point is obtained by imposing that
    $$
    f(t+\epsilon)=\frac{(t+\epsilon)^2}{1+\epsilon_1(t+\epsilon)^2}=t^2+2t\epsilon+\epsilon^2-\epsilon_1 o((t+\epsilon)^4)
    $$
    has no linear $t$ terms locally around zero. This is obtained by imposing $\epsilon=0$. Thus $M_l$ is locally the vanishing of linear coordinates. So it is smooth, in particular it has a reduced structure in $\overline M^Y_{({5})}(\mathbb P ^2,2)$.

    Let us prove the second point. Let us consider the family of stable maps whose image is a conic and centered at the double line $l$. In affine coordinates we may represent it as
    $$
    f(t) = (t^2+\eta_1 t + \eta_2,\eta_3t^2+\eta_4t+\eta_5).
    $$
    Around the inflectional line $l$, the equation of $Y$ is, modulo a change of coordinates, $z_2=z_1^3+o(z_1^4)$. If we want that the family $f(t)$ meets $Y$ with contact order $5$, we need that all coefficients of $t^i$ are zero for $i\le4$ in the following expression:
    $$
    \eta_3t^2+\eta_4t+\eta_5 = (t^2+\eta_1 t + \eta_2)^3.
    $$
    In order to increase the multiplicity, we need that also the coefficient of $t^5$ is zero. That is, $\eta_1=0$. This is true also if we consider conics with one marked point (same as before). Thus, $\sigma$ vanishes around $\mathcal{C}$ with a linear equation. 
\end{proof}
We want to study the moduli space $\overline M^Y_{({6})}(\mathbb P ^2,2)$. As we saw in the proof of Proposition~\ref{prop:M5}, the spaces $M_l$ have contact order $6$. So the following corollary is easily proved.
\begin{cor}\label{cor:spaceM6}
    The moduli space $\overline M^Y_{({6})}(\mathbb P ^2,2)$ has the following connected components.
    \begin{enumerate}[label=(\Alph*)]
        \item One reduced point for every sextactic point of $Y$.
        \item One smooth rational curve $M_l$ for each inflectional line $l$.
    \end{enumerate}
\end{cor}
    
    
In order to compute the number of sextactic conics, we take the degree of the virtual fundamental class of $\overline M^Y_{({6})}(\mathbb P ^2,2)$, minus the contribution of components of type $M_l$. Let $b_d$ denote the contribution of each of component. The rest of the section is dedicated to the computation of $b_d$.

Let us note that this $M_l$ has a natural virtual fundamental class. Indeed, as the line $l$ is fixed, $M_l$ is isomorphic to the subscheme of $\overline M_{0,1}(\mathbb P^1, 2)$ whose contact order at the marked point is $2$. That is,
\begin{equation}\label{eq:Ml_is_M2}
    M_l\cong \overline M^H_{(2)}(\mathbb P^1, 2).
\end{equation}
We use this equation in the proof of the following lemma. 
\begin{lem}\label{lem:b_d}
    We have $b_d=\frac{3}{4}$ for all $d\ge 3$.
\end{lem}
\begin{proof}
    Let us fix an inflectional line $l$. By definition, $b_d$ is the degree of the part of the virtual fundamental class of $\overline M_{(6)}^Y(\mathbb P^2, 2)$ supported on $M_l$. By Example~\ref{example:my_example}, if we restrict Equation~\eqref{eq:fundamental} to $M_l$, we get
    $$
    ((\ev^*(Y)+5\psi)\cdot [\overline M^Y_{(5)}(\mathbb P^2, 2)])_{|M_l}=b_d+3[D(\mathbb P^2, (1,1),(2,3))]_{|M_l}.
    $$
    where
    $$
    D(\mathbb P^2, (1,1),(2,3))=\overline{M}_{0,3}(Y,0)\times_Y \overline{M}^Y_{(2)}(\mathbb P^2, 1)\times_Y \overline{M}^Y_{(3)}(\mathbb P^2, 1).
    $$
    Note that $\overline{M}^Y_{(3)}(\mathbb P^2, 1)$ is zero dimensional, each point is the class of an inflectional line. When restricting to $M_l$, only the component relative to $l$ is not annihilated. Inside $M_l$, $D(\mathbb P^2, (1,1),(2,3))$ is characterized by the fact that it parameterizes the only map with contact order strictly greater than $2$. Using Equation~\eqref{eq:Ml_is_M2}, we get
    \begin{align*}
        [D(\mathbb P^2, (1,1),(2,3))]_{|M_l} &= 
        (\ev^*(Y)+2\psi)_{|M_l} \\
        &= \int_{\overline{M}_{0,1}(\mathbb P^1, 2)} \prod_{k=0}^2 (\ev^*(Y)+k\psi)\\
        &= \frac{1}{2}.
    \end{align*}
    Let $M^*$ denote the component of $\overline M_{(5)}^Y(\mathbb P^2, 2)$ parameterizing birational maps. It is the only component of $\overline M_{(5)}^Y(\mathbb P^2, 2)$, meeting $M_l$, that is not $M_l$ itself. 
    
    By Proposition~\ref{prop:M5restr_and_section}, we have
    $$((\ev^*(Y)+5\psi)\cdot [M^*])_{|M_l}=1,$$
    and, using again Equation~\eqref{eq:Ml_is_M2}, 
    \begin{align*}
        (\ev^*(Y)+5\psi)_{|M_l} &=
        \int_{\overline{M}_{0,1}(\mathbb P^1, 2)} (d\ev^*(H)+5\psi)\cdot \prod_{k=0}^1 (\ev^*(Y)+k\psi)\\
        &=
        \int_{\overline{M}_{0,1}(\mathbb P^1, 2)} 5\ev^*(H)\cdot\psi^2\\
        &= \frac{5}{4}.
    \end{align*}
    This implies that $b_d=1+\frac{5}{4}-\frac{3}{2}=\frac{3}{4}$.
\end{proof}
We are ready to prove the main theorem.
\begin{proof}[Proof of Theorem \ref{thm:main}]
The degree of the virtual class of $\overline M_{(6)}^Y(\mathbb P^2,2)$ is given by the number of sextactic conics $n_d$ plus the contribution from double covers of inflectional lines. By generality, the number of these lines is given by Example~\ref{example:Infl}. Thus
$$
\deg([\overline M_{(6)}^Y(\mathbb{P}^2,2)])=n_d + 3d(d-2)b_d.
$$
Combined with Equation~\eqref{eq:T22} and Lemma~\ref{lem:b_d}, we have $n_d=3d(4d-9)$, as expected.
\end{proof}
\begin{rem}
    When $Y$ is not general, Theorem~\ref{thm:main} is not true even if $Y$ is smooth. Indeed, Lemma~\ref{lem:Gat21} is false in general, so is Corollary~\ref{cor:spaceM6}. Moreover, in Proposition~\ref{prop:M5restr_and_section} we used that, locally around an inflectional point, the equation of $Y$ is $z_2=z_1^3$. This is true because $Y$ is general, otherwise the inflectional point may have greater multiplicity. The case when $Y$ is not general, or even singular, is discussed in \cite{MR1924719,MR4009323}.
\end{rem}
\subsection{Curves of higher degree}
Assuming $\deg(Y)=3$, the multiple cover contribution formula \cite[Proposition~6.1]{MR2667135} computes the contribution from $k$-covers of curves meeting $Y$ with multiplicity $w$,
\begin{equation}\label{eq:BPS}
    \frac{1}{k^2}\binom{k(w-1)-1}{k-1}.
\end{equation}
Lemma~\ref{lem:b_d} coincides with this formula, even without assuming $\deg(Y)=3$. It seems that such a assumption in unnecessary, so we expect that Formula~\eqref{eq:BPS} is still valid if we consider, for example, maximal contact cubics. In this case we have contribution from triple covers of inflectional lines. By Equation~\eqref{eq:T23}, the expected number of such cubics would be
$$
\frac{352}{3}d^2-\frac{812}{3}d-\frac{10}{9}3d(d-2),
$$
so we have the following.
\begin{conj}\label{conj:n3}
    The number of rational cubics with maximal contact with a general smooth plane curve of degree $d\ge3$ is
    \begin{equation}
        N_d = 6d (19d - 44).
    \end{equation}
\end{conj}
Note that $N_3=234$ as claimed in \cite[5.6]{MR4513164}. 
For higher degrees, we expect contribution from covers of irreducible curves, but also from reducible curves obtained by combining maximal contact curves of lower degrees.
\bibliographystyle{amsalpha}
\bibliography{refs}





\end{document}